\documentclass{amsart}
\usepackage{color}
\usepackage[margin=1in]{geometry}
\usepackage{amssymb}
\usepackage{graphicx}
 \usepackage{epstopdf}
 \usepackage{tikz-cd}
\usepackage{hyperref}
\newtheorem{theorem}{Theorem}

\newtheorem{corollary}[theorem]{Corollary}

\theoremstyle{definition}

\newtheorem{proposition}[theorem]{Proposition}
\newtheorem{example}[theorem]{Example}

\theoremstyle{remark}

\numberwithin{equation}{section}
\begin{document}
\title{ITERATIVE ROOTS OF MULTIFUNCTIONS}

\author{B.V. RAJARAMA BHAT}
\address{Indian Statistical Institute, Stat-Math. Unit, R V College Post, Bengaluru 560059, India}
\email{bhat@isibang.ac.in}
\thanks{The first author is supported by  J
    C Bose Fellowship of the Science and Engineering Board, India.}

\author{CHAITANYA GOPALAKRISHNA}
\address{Indian Statistical Institute, Stat-Math. Unit, R V College Post, Bengaluru 560059, India}
\email{cberbalaje@gmail.com, chaitanya\_vs@isibang.ac.in}

\thanks{The second author is
    supported
    by the National Board for Higher Mathematics, India through No:
    0204/3/2021/R\&D-II/7389.}





\begin{abstract}
Some easily verifiable sufficient conditions for the nonexistence of
iterative roots for  multifunctions on arbitrary nonempty sets are
presented. 
Typically if the
graph of the multifunction has a distinguished point with a
relatively large number of paths leading to it then such a
multifunction does not admit any iterative root. These results can
be applied to single-valued maps by considering their pullbacks as
multifunctions. This has been illustrated by showing the
nonexistence of iterative roots of some specified orders for certain
complex polynomials.
\end{abstract}

\maketitle


\section{Introduction}

Given a map $F:X\to X$ on a nonempty set $X$ and an integer $n\ge 1$, the {\it iterative root problem} is to find a map $G:X\to X$ such that functional equation
\begin{eqnarray}\label{it-root}
    G^n=F
\end{eqnarray}
is true on $X$, where $G^n$ is the $n$-th order iterate of $G$ defined recursively by $G^n=G\circ G^{n-1}$ and $G^0={\rm id}$, the identity map on $X$. We call $G$ an $n$-th order iterative root of $F$.
Since the initial works of Babbage \cite{Babbage1815}, Abel \cite{Abel} and K\"{o}nigs \cite{Konigs}, the iterative root problem \eqref{it-root}, which is a weak version of the embedding flow problem \cite{fort1955} of dynamical systems and is applicable to informatics \cite{Kindermann} and neural networks \cite{Iannella}, has been extensively studied in various aspects.
Many of the results are included in the monographs \cite{Kuczma1968,kuczma1990}, the book \cite{Targonski1981}, and the survey papers \cite{Baron-Jarczyk,Targonski1995,zdun-soalrz}. Some of the most recent findings in this direction can also be found in \cite{BG,BGpreprint,Edgar,LiZhang2018,Lin2014,Lin-Zeng-Zhang2017,LiuLiZhang2018,Liu-zhang2021,Yu2021}.

Many researchers
\cite{BGpreprint,Blokh,Humke-Laczkovich1989,simon1989}  have
highlighted the difficulty in solving equation \eqref{it-root}, even
in the class of continuous self-maps of an interval, necessitating
the extension of the iterative root concept to multifunctions. By a
{\it multifunction} (or {\it multivalued function}) on a nonempty
set $X$, we simply mean a function from $X$ to the power set $2^X.$
Given two multifunctions $F$ and $G$ on $X$, their composition
$F\circ G$ is defined by $(F\circ G)(x)=F(G(x))$, where the {\it
image} $F(A)$
of a set $A\subseteq X$ is defined by $F(A)=\cup_{x\in A}F(x)$. Then, as shown in \cite{Jarczyk-Zhang}, the $n$-th order iterate $F^n$ of $F$ is defined recursively by
\begin{eqnarray*}
    F^n(x)=\bigcup_{y\in F^{n-1}(x)}F(y)\quad \text{and}\quad  F^0(x)=\{x\},
\end{eqnarray*}
which satisfy that $F^n(F^m(A))=F^{n+m}(A)=F^m(F^n(A))$. As in
\cite{Hu-Papageorgiou}, a point $x_0\in X$ is said to be a {\it
set-value point} of $F$ if $\#F(x_0)\ge 2$, where $\#A$ denote the
cardinality of a set $A\subseteq X$. Powier\.{z}a
\cite{Powierza1999,Powierza2001,Powierza2002},  and Jarczyk and
Powier\.{z}a \cite{Jarczyk-Powierza},   discussed the existence of
the smallest set-valued iterative roots $G$ of bijections $F$ in the
{\it inclusion sense}, where $F(x)\in G^n(x)$ for all $x\in X$.
Another approach (see \cite{Jarczyk-Zhang,
li2009,LiuLiZhang2021,Lydzinska2018,Zhang-Huang} for example)  is to
look for solutions $G:X\to 2^X$ of the equation \eqref{it-root} for
a multifunction $F$, which is known as the {\it identity sense}. In
\cite{Jarczyk-Zhang, li2009}, some sufficient conditions for the
nonexistence of iterative square roots of multifunctions $F$ on $X$
with exactly one set-value point were obtained,   and in \cite{Lydzinska2018}, some of these results were improved and generalized to deal with order $n$.
The special case where $X$ is a compact interval in the real line $\mathbb{R}$ and $F$ an increasing upper semi-continuous multifunction  was also studied, and results on the construction of iterative roots were presented in \cite{li2009, Lydzinska2018,Nan-li-2011,Zhang-Huang}. However, it appears that no results are found for multifunctions with multiple set-value points, with the exception of the reference \cite{LiuLiZhang2021} for increasing upper semi-continuous multifunctions on compact intervals with finitely many set-value points.



In Section 2 we investigate the equation \eqref{it-root} in the
identity sense for multifuctions $F$ on a general nonempty set $X$.
 The crucial observation we make is
that if there is a point $x_0\in X$ with a large number of  paths
ending at it in the graph of $F$, it creates some kind of bottleneck
in the system $F:X\to 2^X$ (see Figure \ref{Fig3} for example) and
thus the multifunction $F$ will admit no roots of any order. This is
heavily inspired by a similar phenomenon observed for single-valued
maps in \cite{BG, BGpreprint}. We do impose a strong constraint that
the cardinality of $2$-paths ending at $x_0$ is {\em very large} in
comparison to the number of $1$-paths (edges) beginning or ending at
other points, in addition to some mild conditions such as
 $F$ has domain $X$ and that $x_0$ is
not a fixed point of $F$.   Some of the results also demand that $F$
has image $X$. All of this is made precise in Theorem \ref{No-root}
 by quantifying the notion of largeness in various ways. Notably, we put no constraints on the number of
 set valued points for $F$.
 We
present several  examples to show that these results can be applied
in some situations where the known results in \cite{Jarczyk-Zhang,
li2009,Lydzinska2018} fail. 

Multifunctions appear naturally in topology, measure theory and
other fields  while considering inverse images of single-valued
maps. Keeping this in mind, in Section \ref{Sec3}, we introduce
a fundamental class of multifunctions $F$ called {\it pullback
multifunctions}.  An important example of such a multifunction is
the {\it $k$-th root function} $z:=re^{i\theta}\mapsto
z^{1/k}:=\{r^{1/k}e^{i((\theta+2j\pi)/k)}:j=0,1,\ldots, k-1\}$ on
the complex plane $\mathbb{C}$.  The iterative root problem for such
multifunctions can be reduced to the corresponding problem for
single-valued maps. We present several examples to demonstrate
concrete applications of our general results, particularly in determining the nonexistence of iterative roots for
certain complex polynomials.


\section{General multifunctions} \label{Sec2}

In this section we present several sufficient conditions for the nonexistence of iterative roots for multifunctions on arbitrary nonempty sets.
First, we present some  notions and notation required for our discussion.
Let $X$ be a nonempty set and $\mathcal{F}(X)$  consists of all multifunctions $F$ on $X$. For each set $A\subseteq X$, $F\in \mathcal{F}(X)$, and $k\ge 1$, let $F^{-k}(A)$ denote the $k$-th order {\it inverse image} of $A$ by $F$  defined by  $F^{-k}(A)=\{x\in X:F^k(x)\cap A\ne \emptyset\}$.
As in \cite[p.34]{Aubin-Frankowska}, the {\it domain} and {\it image} of an $F\in \mathcal{F}(X)$ are defined by
\begin{eqnarray*}
    {\rm Dom}(F)=\{x\in X:\#F(x)\ge 1\}~\quad \text{and}~\quad {\rm Im}(F)=F(X).
\end{eqnarray*}

As in \cite[p.34]{Aubin-Frankowska},  we define the {\it graph} of any $F\in \mathcal{F}(X)$ as the directed graph $\mathcal{G}_F=(X, E_F)$, with the vertex set $X$ and the edge set $E_F=\{(x,y): y\in F(x)\}$. For  $x,y \in X$ and $k\in \mathbb{N}$, by a {\it $k$-path} (or {\it length $k$  path}) in $\mathcal{G}_F$ from $x$ to $y$, we mean a sequence $((x,u_1), (u_1,u_2), \ldots, (u_{k-2}, u_{k-1}), (u_{k-1},y))$ of edges (not necessarily distinct) that joins a sequence of vertices $(x,u_1,\ldots,u_{k-1},y)$.
For each $F, G\in \mathcal{F}(X)$, $x,y,z \in X$, $A,B\subseteq X$ and $k, l\in \mathbb{N}$, let
\begin{align*}
    P_F(x,y;k)&:=\{p: p~\text{is a}~k\text{-path in}~\mathcal{G}_F~\text{from}~x~\text{to}~y\},\\
    P_F(A,y;k)&:=\{p: p~\text{is a}~k\text{-path in}~\mathcal{G}_F~\text{that begins at a point in}~A~\text{and  ends at}~y\},\\
    P_F(x,A;k)&:=\{p: p~\text{is a}~k\text{-path in}~\mathcal{G}_F~\text{that begins at}~x~\text{and ends at a point in}~A\},\\
    P_F(A,B;k)&:=\{p: p~\text{is a}~k\text{-path in}~\mathcal{G}_F~\text{that begins at a point in}~A~\text{and ends at a point in}~B\},
\end{align*}
and
\begin{align*}
    P_F(x,y;k)\bigvee P_G(y,z;l) &:=\{p\vee q: p\in     P_F(x,y;k)~\text{and}~q\in P_G(y,z;l)\},
\end{align*}
where for any two paths $$p:=((x,u_1), (u_1,u_2), \ldots, (u_{k-2},
u_{k-1}), (u_{k-1},y)) \in P_F(x,y;k)$$ and $$q:=((y,v_1),
(v_1,v_2), \ldots, (v_{l-2}, v_{l-1}), (v_{l-1},z))\in P_G(y,z;l),$$
$p\vee q$ is the $(k+l)$-path in $\mathcal{G}_F\cup\mathcal{G}_G$ obtained by concatenating them, defined by
\begin{eqnarray*}
    p\vee q=((x,u_1), (u_1,u_2), \ldots, (u_{k-2}, u_{k-1}), (u_{k-1},y), (y,v_1), (v_1,v_2), \ldots,(v_{l-2}, v_{l-1}), (v_{l-1},z)).
\end{eqnarray*}
 Then it is easy to see that
\begin{eqnarray*}
    P_F(A,y;k)=\bigcup_{x\in A}P_F(x,y;k),\quad  P_F(x,A;k)=\bigcup_{y\in A}P_F(x,y;k),
\end{eqnarray*}
\begin{eqnarray*}
     P_F(A,B;k)=\bigcup_{y\in B}P_F(A,y;k)=\bigcup_{x\in A}P_F(x,B;k)
\end{eqnarray*}
and
\begin{eqnarray*}
    P_F(x,y;k) \bigvee P_G(y,z;l)=P_{G\circ F}(x,z;k+l)
\end{eqnarray*}
for all $F, G\in \mathcal{F}(X)$, $x,y,z \in X$, $A,B\subseteq X$
and $k, l\in \mathbb{N}$.  Let
\begin{align*}
    \mathcal{F}_{M}(X)&:=\{F\in \mathcal{F}(X):\#P_F(x,X;1)\le M~\text{for all}~x\in X\}~ \text{for each}~M\in \mathbb{N}, \\
    \mathcal{F}_{\textsf{f}}(X)&:=\{F\in \mathcal{F}(X):P_F(x,X;1)~\text{is finite for all}~x\in X\}, \\
    \mathcal{F}_{\textsf{c}}(X)&:=\{F\in \mathcal{F}(X):P_F(x,X;1)~\text{is countable for all}~x\in
    X\}.
\end{align*}
It is useful to note that $\#P_F(X,x;1)=\#F^{-1}(\{x\})$ and $\#P_F(X,x;k)\ge\#F^{-k}(\{x\})$, and similarly $\#P_F(x,X;1)=\#F(x)$ and $\#P_F(x,X;k)\ge\#F^k(x)$ for all $x\in X$ and $k\ge 2$.

Now we present a theorem which can be considered the main result of
this article. It is motivated by analogous results for single-valued
maps in \cite{BG, BGpreprint}.  We observe that, generally, if a
multifunction $F$ has a special point $x_0$ with relatively large
number of paths leading to it, then $F$ does not admit any iterative
root. Actually, we only compare the number of paths of length two
reaching $x_0$ with the number of edges at other points.   The
comparison is made precise in the following theorem.

\begin{theorem}\label{No-root}
 Let $F$ be a multifunction on $X$ such that ${\rm Dom}(F)=X$ and $x_0 \notin F(x_0)$ for some $x_0\in X$.
    \begin{enumerate}

 \item[\bf (1)]{\rm (Large finite sets vs small finite sets)} Suppose that $F$ satisfies {\bf (1.a)} $\# P_F(X,x_0;2)>MN^3$
  and {\bf (1.b)}  $\# P_F(X,x;1)\le N$ for all $x\ne x_0$ in
        $X$ and for some $M, N \in {\mathbb N}$.
            %
            %
        Then $F$ has no iterative roots
        of order $n\ge 2$ in $\mathcal{F}_{M}(X)$.
        If, in addition, $F\in \mathcal{F}_{M}(X)$ and ${\rm Im}(F)=X$, then $F$ has no iterative roots of order $n\ge 2$ at all.

  \item[\bf (2)] {\rm (Infinite sets vs finite sets)} Suppose that $F$ satisfies {\bf (2.a)} $P_F(X,x_0;2)$ is infinite, and {\bf (2.b)}  $ P_F(X,x;1)$ is finite for
        all $x\ne x_0$ in $X$.
        Then $F$ has no iterative roots of order $n\ge 2$ in $\mathcal{F}_{\textsf{f}}(X)$. If, in addition, $F\in \mathcal{F}_{\textsf{f}}(X)$ and ${\rm Im}(F)=X$, then $F$ has no iterative roots of order $n\ge 2$ at all.

  \item[\bf (3)] {\rm (Uncountable sets vs countable sets)} Suppose that $F$ satisfies {\bf (3.a)} $P_F(X,x_0;2)$ is uncountable, and {\bf (3.b)}  $ P_F(X,x;1)$ is countable for
        all $x\ne x_0$ in $X$.
        Then $F$ has no iterative roots of order $n\ge 2$ in $\mathcal{F}_{\textsf{c}}(X)$. If, in addition, $F\in \mathcal{F}_{\textsf{c}}(X)$ and ${\rm Im}(F)=X$, then $F$ has no iterative roots of order $n\ge 2$ at all.
    \end{enumerate}
\end{theorem}
\begin{proof}
    The key to the proof is to estimate $P_G(G^{n-1}(X), x_0;n+1)$ in two different ways, where $G$ is an $n$-th root of $F$,
    by observing $G^{n+1}=F\circ G=G\circ F$.
    To prove the first part of result {\bf (1)}, suppose,
    on the contrary, that $F$ has an iterative root $G$ of some order $n\ge 2$ in $\mathcal{F}_{M}(X)$.
    Since $x_0\notin F(x_0)$, it is clear that $x_0\notin G(x_0)$. We first establish the following two inequalities that we use to estimate $P_G(G^{n-1}(X), x_0;n+1)$.
    \begin{enumerate}
        \item[\bf (i)] $    \#P_G(G^{n-2}(X),x_0;1)\le NM$.

        \item[\bf (ii)] $\# P_G(X,y;n-1)\le N$ whenever   $y\in G^{n-1}(X)\cap G^{-1}(\{x\})$ and $x\ne x_0$.
    \end{enumerate}

Since $P_G(X,x_0;n-1)\ne \emptyset$ by  {\bf (1.a)}, and $P_G(x_0,X;1)\ne \emptyset$ as ${\rm Dom}(F)=X$,
we have $P_G(X,x_0;n-1) \bigvee P_G(x_0,X;1)\ne \emptyset$. Also, since
    \begin{align*}
        P_G(X,x_0;n-1)&=P_G(X,G^{n-2}(X);n-2)\bigvee P_G(G^{n-2}(X),x_0;1),
    \end{align*}
    we have
    \begin{align}\label{01}
        \# P_G(X,x_0;n-1)\ge \#P_G(G^{n-2}(X),x_0;1).
    \end{align}
   In particular, we  see  from \eqref{01} that
        \begin{align}
         \#\left(   P_G(X,x_0;n-1) \bigvee P_G(x_0,X;1)\right) &\ge \#P_G(G^{n-2}(X),x_0;1). \label{02}
    \end{align}
Further, we have
    \begin{align}\label{03}
    P_G(X,x_0;n-1) \bigvee P_G(x_0,X;1)&=   P_G(X,x_0;n-1)\bigvee P_G(x_0,G(x_0);1) \nonumber \\
    &\subseteq  P_G(X,G(x_0);n) \nonumber \\
    & =  P_F(X,G(x_0);1) \nonumber \\
    &=\bigcup_{y\in G(x_0)}P_F(X,y;1).
\end{align}
    Therefore, since $\#G(x_0)\le M$ as $G\in \mathcal{F}_{M}(X)$,  by using \eqref{02} and {\bf (1.b)}
    we obtain
    \begin{align*}
        \#P_G(G^{n-2}(X),x_0;1)
        &\le  \sum_{y\in G(x_0)}\# P_F(X,y;1)
        \le N \cdot \#G(x_0) \le NM.
    \end{align*}
This proves {\bf (i)}.
    Since
    \begin{align}\label{04}
        P_F(X,x;1)&=  \bigcup_{y\in G^{n-1}(X)\cap G^{-1}(\{x\})}P_G(X,y;n-1)\bigvee P_G(y,x;1)
    \end{align}
for $x\ne x_0$ with $F^{-1}(\{x\})\ne \emptyset$,   by using {\bf (1.b)} we have
    \begin{align*}
        \# P_G(X,y;n-1)\le \#\left(P_G(X,y;n-1)\bigvee P_G(y,x;1)\right) \le \# P_F(X,x;1) \le N
    \end{align*}
    whenever   $y\in G^{n-1}(X)\cap G^{-1}(\{x\})$ and $x\ne x_0$. This proves {\bf (ii)}.

    We now estimate $P_G(G^{n-1}(X), x_0;n+1)$ in two different ways to get a contradiction, as shown below.
    When $G^{n+1}=G\circ F$, we have
    \begin{align}\label{GF}
        P_G(X,x_0;n+1)&= \bigcup_{y\in F(X)\cap G^{-1}(\{x_0\})}P_F(X,y;1)\bigvee P_G(y,x_0;1).
    \end{align}
    Then by using {\bf (1.b)} and {\bf (i)}  we see that
    \begin{align*}
        \#  P_G(X,x_0;n+1) &\le N\sum_{y\in F(X)\cap G^{-1}(\{x_0\})} \# P_G(y,x_0;1) \nonumber\\
        &= N\cdot \#P_G(F(X)\cap G^{-1}(\{x_0\}),x_0;1) \nonumber\\
        &\le N\cdot \#P_G(F(X),x_0;1) \nonumber\\
        &= N\cdot \#P_G(G^n(X),x_0;1) \nonumber\\
        &\le  N\cdot \#P_G(G^{n-2}(X),x_0;1)\\
        &\le N\cdot NM=MN^2.
    \end{align*}
Therefore, as $G^{n-1}(X)\subseteq X$, we have
\begin{align}\label{C-GF}
     \#  P_G(G^{n-1}(X),x_0;n+1) \le MN^2.
\end{align}
    On the other hand, when $G^{n+1}=F\circ G$, we have
    \begin{align}
        P_F(X,x_0;2)&=P_F(X,F(X);1)\bigvee P_F(F(X),x_0;1) \nonumber \\
        &=\left(P_G(X,G^{n-1}(X);n-1)\bigvee P_G(G^{n-1}(X),F(X);1)\right)\bigvee P_F(F(X),x_0;1) \nonumber \\
        &=P_G(X,G^{n-1}(X);n-1)\bigvee \left(P_G(G^{n-1}(X),F(X);1)\bigvee P_F(F(X),x_0;1) \right) \nonumber \\
        &=\bigcup_{x\in  F^{-1}(\{x_0\})}~\bigcup_{y\in G^{n-1}(X)\cap G^{-1}(\{x\})}P_G(X,y;n-1)\bigvee \left( P_G(y,x;1) \bigvee P_F(x,x_0;1)\right). \label{FG}
    \end{align}
    Therefore, by using {\bf (1.a)} and {\bf (ii)} we get
    \begin{align*}
        MN^3 < \# P_F(X,x_0;2)
        &\le N \sum_{\substack{x\in  F^{-1}(\{x_0\})\\ y\in G^{n-1}(X)\cap G^{-1}(\{x\})}}\#\left(P_G(y,x;1) \bigvee P_F(x,x_0;1)\right)\\
        &=N\sum_{\substack{y\in G^{n-1}(X)\cap G^{-(n+1)}(\{x_0\})}}\#P_G(y,x_0;n+1)\\
        & = N\cdot \#P_G(G^{n-1}(X)\cap G^{-(n+1)}(\{x_0\}),x_0;n+1)\nonumber\\ &\le N\cdot\#P_G(G^{n-1}(X),x_0;n+1),
    \end{align*}
    implying that
    \begin{align}\label{C-FG}
        \#P_G(G^{n-1}(X),x_0;n+1)>MN^2.
    \end{align}
    Thus, a contradiction obtained from \eqref{C-GF} and \eqref{C-FG} shows that $F$ has no iterative roots of order $n\ge 2$ in $\mathcal{F}_{M}(X)$.

    Now,
    to prove the second part, let $F\in \mathcal{F}_{M}(X)$ with ${\rm Im}(F)=X$, and suppose, on the contrary, that $G\in \mathcal{F}(X)$ is an iterative root of $F$ of some order $n\ge 2$. Then ${\rm Im}(G)=X$, and
    by the first part we have  $\# P_G(\tilde{x},X;1)>M$ for some $\tilde{x}\in X$.
    Let $\tilde{x}\in G^{n-1}(\tilde{y})$ for some $\tilde{y}\in X$ that exists because  $G^{n-1}(X)=X$. Then, since
    \begin{eqnarray}\label{ftildey}
        F(\tilde{y})=G(G^{n-1}(\tilde{y}))\supseteq G(\tilde{x}),
    \end{eqnarray}
    we obtain $\#P_F(\tilde{y},X;1)>M$, which is a contradiction to our assumption that $F\in \mathcal{F}_M(X)$. Therefore $F$ has no iterative roots of order $n\ge 2$ at all, completing the  proof of result  {\bf (1)}.

    Next, to prove {\bf (2)},  suppose that $F$ has an iterative root $G$ of some order $n\ge 2$ in $\mathcal{F}_{\textsf{f}}(X)$.  Then, using the similar arguments as above, it is clear that $x_0\notin G(x_0)$, $P_G(X,x_0;n-1) \bigvee P_G(x_0,X;1)\ne \emptyset$, and equations \eqref{01}, \eqref{02}, \eqref{03}, \eqref{04}, \eqref{GF} and \eqref{FG} are  satisfied.
    Additionally,  since $P_G(x_0,X;1)$ is
 finite as $G\in \mathcal{F}_{\textsf{f}}(X)$, and $P_F(X,y;1)$ is finite for all $y\ne x_0$ by {\bf (2.b)},  we see from \eqref{03} that $P_G(X,x_0;n-1) \bigvee P_G(x_0,X;1)$ is finite, which implies  by \eqref{02} that $P_G(G^{n-2}(X),x_0;1)$ is finite. Further, it follows from \eqref{04} and {\bf (2.b)} that
    \begin{align}\label{05}
        P_G(X,y;n-1)~\text{is finite if}~y\in G^{n-1}(X)\cap G^{-1}(\{x\})~\text{and}~x\ne x_0.
    \end{align}

    We now estimate  $\#P_G(G^{n-1}(X),x_0;n+1)$ in two different ways, as above, to find a contradiction.
    When $G^{n+1}=G\circ F$, by using \eqref{GF}, {\bf (2.b)} and the fact that $P_G(G^{n-2}(X),x_0;1)$ is finite, we see that $P_G(G^{n-1}(X),x_0;n+1)$ is finite. On the other hand, when $G^{n+1}=F\circ G$, by using \eqref{FG}, \eqref{05} and {\bf (2.a)}, it follows that $P_G(G^{n-1}(X),x_0;n+1)$ is infinite. Therefore $F$ has no iterative roots of order $n\ge 2$ in $\mathcal{F}_{\textsf{f}}(X)$.

    In order to prove the second part, let $F\in \mathcal{F}_{\textsf{f}}(X)$ with ${\rm Im}(F)=X$,  and assume, on the contrary, that $G\in \mathcal{F}(X)$ is an iterative root of $F$ of some order $n\ge 2$. Then ${\rm Im}(G)=X$, and
    by the first part we have  $P_G(\tilde{x},X;1)$ is infinite for some $\tilde{x}\in X$. Since  $G^{n-1}(X)=X$,
    we have $\tilde{x}\in G^{n-1}(\tilde{y})$ for some $\tilde{y}\in X$. Then, by \eqref{ftildey} it follows that $P_F(\tilde{y},X;1)$ is infinite,  contradicting  our assumption that $F\in \mathcal{F}_{\textsf{f}}(X)$. Therefore $F$ has no iterative roots of order $n\ge 2$ at all.  This completes the proof of result {\bf (2)}.

    The proof of result {\bf (2)}  is based on the fact that a finite union of finite sets is finite.
    The proof of result {\bf (3)} is similar to that of {\bf (2)}, using the result that a countable union of countable sets is countable.
\end{proof}

As a consequence of the above theorem, we have the following
corollary. 
Instead of counting paths, we now count the number of points in certain sets.

\begin{corollary}\label{C1}
     Let $F$ be a multifunction on $X$ such that ${\rm Dom}(F)=X$ and $x_0 \notin F(x_0)$ for some $x_0\in X$.
    \begin{enumerate}
        \item[\bf (1)] {\rm (Large finite sets vs small finite sets)}
        Suppose that $F$ satisfies
          $\# F^{-2}(\{x_0\})>MN^3$  and   $\# F^{-1}(\{x\})\le N$  for
        all  $x\ne x_0$ in $X$ and for some $N, M\in \mathbb{N}$.
            %
            %
        Then $F$ has no iterative roots
        of order $n\ge 2$ in $\mathcal{F}_{M}(X)$.
        If, in addition, $F\in \mathcal{F}_{M}(X)$ and ${\rm Im}(F)=X$, then $F$ has no iterative roots of order $n\ge 2$ at all.

        \item[\bf (2)] {\rm (Infinite sets vs finite sets)} Suppose that $F$ satisfies  $ F^{-2}(\{x_0\})$ is infinite and   $ F^{-1}(\{x\})$ is finite for
        all $x\ne x_0$ in $X$.
        Then $F$ has no iterative roots of order $n\ge 2$ in $\mathcal{F}_{\textsf{f}}(X)$. If, in addition, $F\in \mathcal{F}_{\textsf{f}}(X)$ and ${\rm Im}(F)=X$, then $F$ has no iterative roots of order $n\ge 2$ at all.

        \item[\bf (3)] {\rm (Uncountable sets vs countable sets)} Suppose that $F$ satisfies $ F^{-2}(\{x_0\})$ is uncountable and   $ F^{-1}({x})$ is countable for
        all $x\ne x_0$ in $X$.  Then $F$ has no iterative roots of order $n\ge 2$ in $\mathcal{F}_{\textsf{c}}(X)$. If, in addition, $F\in \mathcal{F}_{\textsf{c}}(X)$ and ${\rm Im}(F)=X$, then $F$ has no iterative roots of order $n\ge 2$ at all.
    \end{enumerate}
\end{corollary}
\begin{proof}
    Follows from Theorem \ref{No-root}, because $\# P_F(X, x;2)\ge \#F^{-2}(\{x\})$ and $\# P_F(X,x;1)=\# F^{-1}(\{x\})$ for all $x\in X$ and $F\in \mathcal{F}(X)$.
\end{proof}

We remind that the above results
and those  that follow do not demand any restriction on the
cardinality of the set-value points of the map under consideration.
Therefore, unlike the known results in \cite{Jarczyk-Zhang,
li2009,Lydzinska2018}, these ones deal with multifunctions with
set-value points of arbitrary cardinality. Furthermore, in the
special case where $M=1$, the above corollary reduces to Theorem 2
in \cite{BGpreprint}, valid for single-valued maps. It is worth
noting, however, that the proof presented here is both different and
simpler than that presented there. We now illustrate
the above results with some examples.

\begin{example}\label{E1}
    {\rm Let
        \begin{eqnarray*}
            X=\{x_{i}:i\ge 0\}\bigcup\{x_{-1}^{(j)}:1\le j\le 4\}\bigcup\{x_{-i}^{(j)}:i\ge 2~\text{and}~j=1,2\}
        \end{eqnarray*}
          and  $F_1:X\to 2^X$ be defined by
        \begin{eqnarray*}
            F_1(x_{i})&=&\{ x_{i+1}\}\quad \text{for}~i\ge 0,\\
            F_1(x_{-1}^{(j)})&=&\{x_0\}\quad \text{for}~1\le j\le 4,\\
            F_1(x_{-2}^{(j)})&=& \{x_{-1}^{(2j-1)}, x_{-1}^{(2j)}\} \quad \text{for}~j=1,2,\\
            F_1(x_{-i}^{(j)})&=&\{x_{-(i-1)}^{(j)}\}\quad \text{for}~i\ge 3~\text{and}~j=1,2
        \end{eqnarray*}
        (see Figure \ref{Fig3}).  Then $F_1$ has no iterative roots of order $n\ge 2$ by result {\bf (1)} of  Theorem \ref{No-root}  with $N=1$ and $M=2$. Further, it is clear that Corollary \ref{C1} is not applicable, because $\#F^{-2}(\{x_0\})=2=MN^3$.}
\end{example}
\begin{figure}[h]
    \centering
    \includegraphics[scale=0.6]{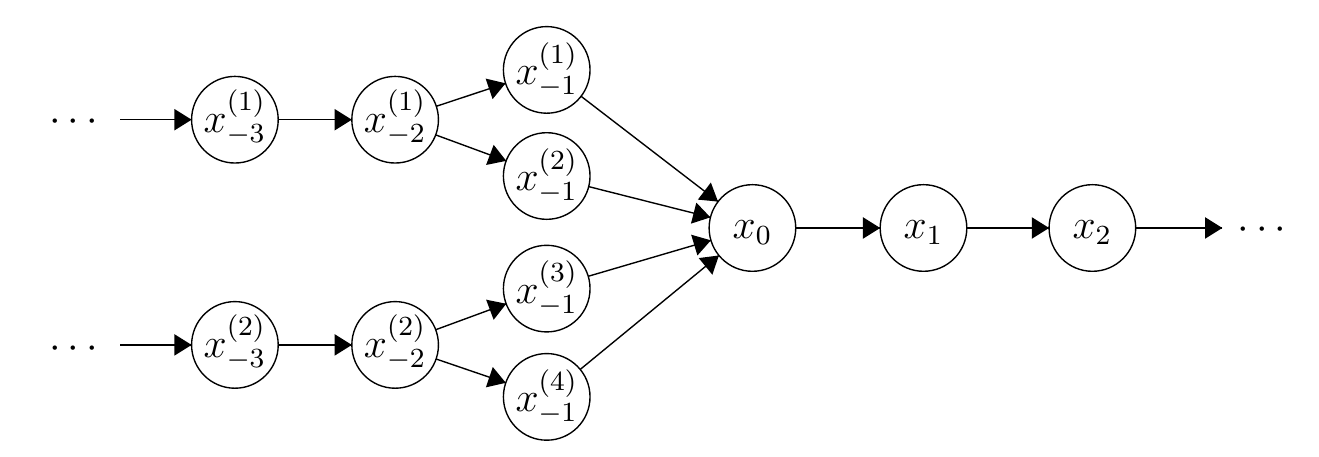}
    \caption{$F_1$}
    \label{Fig3}
\end{figure}

\begin{example}
    {\rm    Consider
        \begin{eqnarray*}
            X=\{x_0\}\bigcup\{x_{i}^{(j)}:i\ge 1~\text{and}~j=1,2\}\bigcup\{x_{-i}^{(j)}:j\ge 1~\text{and}~j=1,2,3\}
        \end{eqnarray*}
          and let   $F_2:X\to 2^X$ be defined by
        \begin{eqnarray*}
            F_2(x_0)&=&\{x_{1}^{(1)},x_{1}^{(2)}\},\\
            F_2(x_{i}^{(j)})&=&\{ x_{i+1}^{(j)}\}\quad \text{for}~i\ge 1~\text{and}~j=1,2,\\
            F_2(x_{-1}^{(j)})&=&\{x_0\}\quad \text{for}~j=1,2,3,\\
            F_2(x_{-i}^{(j)})&=&\{x_{-(i-1)}^{(j)}\}\quad \text{for}~i\ge 2~\text{and}~j=1,2,3
        \end{eqnarray*}
        (see Figure \ref{Fig0}). Then $F_2$ has no iterative roots of order $n\ge 2$ by result {\bf (1)} of  Corollary \ref{C1} (or Theorem \ref{No-root}) with $N=1$ and $M=2$.

        \begin{figure}[h]
            \centering
            \includegraphics[scale=0.6]{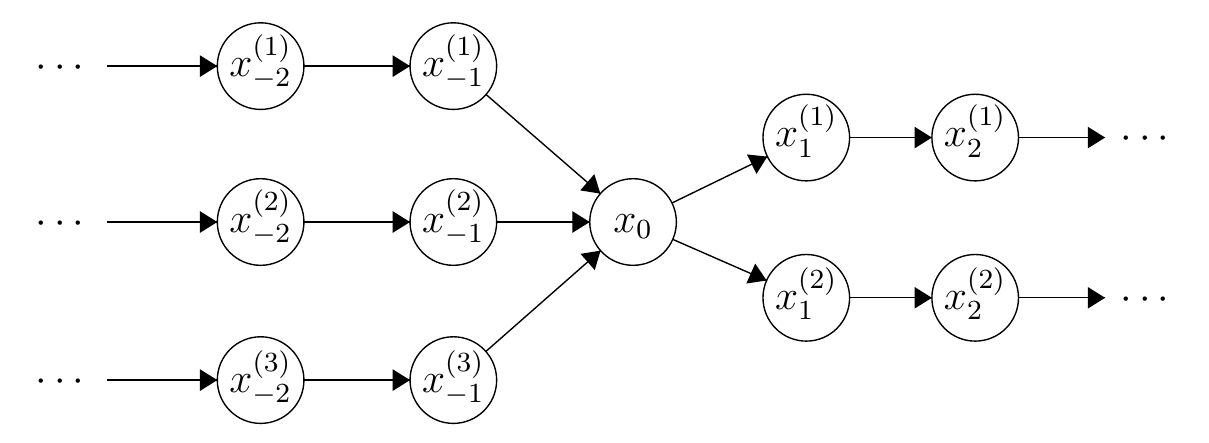}
            \caption{$F_2$}
            \label{Fig0}
        \end{figure}
    }
\end{example}

\begin{example}\label{EX1}
    {\rm Consider the multifunction $F_3:[0,1]\to 2^{[0,1]}$ given by
        \begin{eqnarray*}
            F_3(x)=\left\{\begin{array}{cll}
                \{0,1\}&\text{if}&x=0,\\
                \{\frac{1}{4}-\frac{x}{2}\} &\text{if}&x\in (0,\frac{1}{4}],\\
                \{\frac{1}{8}\} &\text{if}&x\in[\frac{1}{4},\frac{1}{2}],\\
                \{2x-1\}    &\text{if}& x\in (\frac{1}{2},1]
            \end{array}\right.
        \end{eqnarray*}
        with exactly one set-value point $c=0$  (see Figure \ref{Fig1}). Since $\{c\}$ is not a value of $F_3$, not only Theorems 1 and 2 of \cite{Jarczyk-Zhang} and Theorems 1 of \cite{li2009} for iterative square roots, but also Theorem 2 of \cite{Lydzinska2018} for iterative $n$-th roots,
        do not work.
        Furthermore, since $F_3(c)=\{0,1\}$ such that $F_3(1)=\{1\}$, Theorem 2 of \cite{li2009} is also not applicable.
        However, $F_3$ has no iterative roots of order $n\ge 2$ by result {\bf (2)} of  Theorem \ref{No-root} (or Corollary \ref{C1}) with $x_0=1/8$.

        \begin{figure}[h]
 \centering
\includegraphics[scale=1]{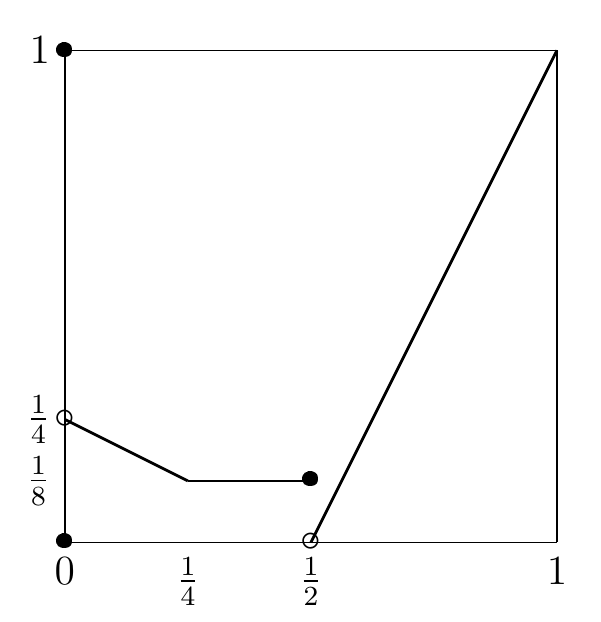}
            \caption{$F_3$}
            \label{Fig1}
            \end{figure}
    }
\end{example}

\begin{example}
    {\rm Consider the multifunction $F_4:[0,1]\to 2^{[0,1]}$ given by
        \begin{eqnarray*}
            F_4(x)=\left\{\begin{array}{cll}
                \{\frac{3}{4},1\}&\text{if}&x=0,\\
                \{\frac{1}{4}-x\}   &\text{if}&x\in (0,\frac{1}{4}],\\
                \{0\}   &\text{if}&x\in[\frac{1}{4},\frac{1}{2}),\\
                \{0\}\cup([\frac{3}{4},1]\cap \mathbb{Q})   &\text{if}&x=\frac{1}{2},\\
                \{2x-1\}    &\text{if}& x\in (\frac{1}{2},1]
            \end{array}\right.
        \end{eqnarray*}
        with the set-value points $0$ and $1/2$ (see Figure \ref{Fig2}).
        Since $F_4$ has more than one set-value point, none of the known results mentioned in Example \ref{EX1} apply. However, $F_4$
        has no iterative roots of order $n\ge 2$ by result {\bf (3)} of Theorem \ref{No-root} (or Corollary \ref{C1}) with $x_0=0$.
    }
\end{example}
\begin{figure}[h]
  \centering
\includegraphics[scale=1]{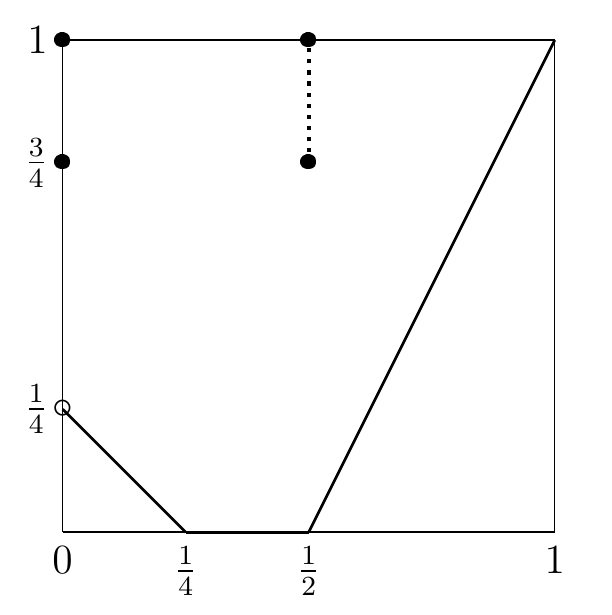}
    \caption{$F_4$}
    \label{Fig2}
\end{figure}

It is worth noting that when $M=1$, the bound $MN^3$ considered in result {\bf (1)} of Theorem \ref{No-root} is optimal, as demonstrated by the examples in {\bf (iv)} of Section 4 in \cite{BGpreprint}.
Additionally, under the assumption of the axiom of choice, the results {\bf (2)} and {\bf (3)} of Theorem \ref{No-root} can be further generalized in the context of infinite cardinal numbers to provide many more similar results for the nonexistence of iterative roots.
 More precisely, we have the following result, where $\aleph_\alpha$'s are precisely the infinite cardinal numbers indexed by the  ordinal numbers $\alpha$, $\le$ is the order among cardinal numbers $\aleph_\alpha$, and
\begin{eqnarray*}
    \mathcal{F}_{\aleph_\alpha}(X):=\{F\in \mathcal{F}(X):\#P_F(x,X;1)<\aleph_{\alpha}~\text{for all}~x\in X\}
\end{eqnarray*}
for each ordinal number $\alpha$.

\begin{theorem}\label{Gen-card} {\rm (Sets of cardinality at least $\aleph_\alpha$ vs those less than $\aleph_\alpha$)}
     Let $F$ be a multifunction on $X$ such that ${\rm Dom}(F)=X$ and $x_0 \notin F(x_0)$ for some $x_0\in X$. Further, suppose that $F$ satisfies $\#P_F(X, x_0;2)\ge \aleph_{\alpha}$ and  $ \#P_F(X,x;1)<\aleph_\alpha$  for
    all $x\ne x_0$ in $X$ and for some infinite cardinal number $\aleph_\alpha$.     Then $F$ has no iterative roots of order $n\ge 2$ in $\mathcal{F}_{\aleph_{\alpha}}(X)$. If, in addition, $F\in \mathcal{F}_{\aleph_{\alpha}}(X)$ and ${\rm Im}(F)=X$, then $F$ has no iterative roots of order $n\ge 2$ at all.
\end{theorem}

\begin{proof}
    The proof of  result {\bf (2)} of Theorem \ref{No-root} is based on the fact that a
    finite union of finite sets is finite. The proof of this result is similar, using the result that a union of a collection of cardinality $\aleph_{\alpha}$ of sets of cardinality $\aleph_\alpha$ has cardinality $\aleph_\alpha$.
\end{proof}

As a consequence of the above theorem,  we have the following result, with proof similar to that of Corollary \ref{C1}.
\begin{corollary}\label{C2}
    {\rm (Sets of cardinality at least $\aleph_\alpha$ vs those less than $\aleph_\alpha$)}
     Let $F$ be a multifunction  on $X$ such that ${\rm Dom}(F)=X$ and $x_0 \notin F(x_0)$ for some $x_0\in X$. Further, suppose that $F$ satisfies
    $\#F^{-2}(\{x_0\})\ge \aleph_{\alpha}$ and   $ \#F^{-1}(\{x\})<\aleph_{\alpha}$  for
    all $x\ne x_0$ in $X$  and for some infinite cardinal number $\aleph_\alpha$.     Then $F$ has no iterative roots of order $n\ge 2$ in $\mathcal{F}_{\aleph_{\alpha}}(X)$. If, in addition, $F\in \mathcal{F}_{\aleph_{\alpha}}(X)$ and ${\rm Im}(F)=X$, then $F$ has no iterative roots of order $n\ge 2$ at all.
\end{corollary}



\section{Inverse and Pullback multifunctions}\label{Sec3}

To describe the continuity or measurability of single-valued maps, we look at their
inverse images. We may think of them as
multivalued functions. It is clearly a very important class, and we
refer to it as pullback multifunctions. Unfortunately they are not
covered by the results of the previous section due to some trivial reasons
suggesting that we should be analyzing the inverses of general
multifunctions. The graph of the inverse of a multifunction is obtained
simply by reversing arrows/directions of the original graph. Almost
all the results of the previous section translate  to this setting, and we
present them here first without proofs. Then we specialize to pullback
multifunctions and see the implications of these results.

 As in \cite[p.34]{Aubin-Frankowska}, the {\it inverse} of any $F\in \mathcal{F}(X)$ is the multifunction $F^{-1}\in \mathcal{F}(X)$ defined by $x\in F^{-1}(y)$ if $y\in F(x)$. Then $\mathcal{G}_{F^{-1}}=\{(y,x):(x,y)\in \mathcal{G}_F\}$. Furthermore, $F=G_1\circ G_2$ if and only if $F^{-1}=G_2^{-1}\circ G_1^{-1}$ for all multifunctions $G_1, G_2\in \mathcal{F}(X)$ as shown in \cite[p.37]{Aubin-Frankowska}, implying that $F$ has an iterative root of order $n$ in $\mathcal{E}\subseteq \mathcal{F}(X)$ if and only if $F^{-1}$ has an iterative root of order $n$ in $\mathcal{E}^{-1}:=\{H:H^{-1}\in \mathcal{E}\}$. Consequently, we can deduce the following corollary from Theorem \ref{No-root} and Corollary \ref{C1}, where
 \begin{align*}
    \mathcal{F}^{-1}_{M}(X)&:=\{F\in \mathcal{F}(X):F^{-1}\in \mathcal{F}_M(X)\}~ \text{for each}~M\in \mathbb{N}, \\
    \mathcal{F}^{-1}_{\textsf{f}}(X)&:=\{F\in \mathcal{F}(X):F^{-1}\in \mathcal{F}_{\textsf{f}}(X)\}, \\
    \mathcal{F}^{-1}_{\textsf{c}}(X)&:=\{F\in \mathcal{F}(X):F^{-1}\in \mathcal{F}_{\textsf{c}}(X)\},\\ \mathcal{F}^{-1}_{\aleph_\alpha}(X)&:=\{F\in \mathcal{F}(X):F^{-1}\in \mathcal{F}_{\aleph_\alpha}(X)\}~ \text{for each ordinal number}~\alpha,
 \end{align*}
and we assume the axiom of choice in result {\bf (4)}.
\begin{corollary}\label{C3}
     Let $F$ be a multifunction on $X$ such that ${\rm Im}(F)=X$ and $x_0 \notin F(x_0)$ for some $x_0\in X$.
        \begin{enumerate}
        \item[\bf (1)] {\rm (Large finite sets vs small finite sets)}
        Suppose that $F$ satisfies $\# P_F(x_0,X;2)>MN^3$  and   $\# P_F(x,X;1)\le N$  for
        all  $x\ne x_0$ in $X$ and for some $N, M\in \mathbb{N}$.
            %
            %
        Then $F$ has no iterative roots
        of order $n\ge 2$ in $\mathcal{F}^{-1}_{M}(X)$.
        If, in addition, $F\in \mathcal{F}^{-1}_{M}(X)$ and ${\rm Dom}(F)=X$, then $F$ has no iterative roots of order $n\ge 2$ at all.

        \item[\bf (2)] {\rm (Infinite sets vs finite sets)} Suppose that $F$ satisfies  $P_F(x_0,X;2)$ is infinite and  $ P_F(x,X;1)$ is finite for
        all $x\ne x_0$ in $X$.
        Then $F$ has no iterative roots of order $n\ge 2$ in $\mathcal{F}^{-1}_{\textsf{f}}(X)$. If, in addition, $F\in \mathcal{F}^{-1}_{\textsf{f}}(X)$ and ${\rm Dom}(F)=X$, then $F$ has no iterative roots of order $n\ge 2$ at all.

        \item[\bf (3)] {\rm (Uncountable sets vs countable sets)} Suppose that $F$ satisfies  $P_F(x_0,X;2)$ is uncountable and  $ P_F(x,X;1)$ is countable for
        all $x\ne x_0$ in $X$.
        Then $F$ has no iterative roots of order $n\ge 2$ in $\mathcal{F}^{-1}_{\textsf{c}}(X)$. If, in addition, $F\in \mathcal{F}^{-1}_{\textsf{c}}(X)$ and ${\rm Dom}(F)=X$, then $F$ has no iterative roots of order $n\ge 2$ at all.

        \item[\bf (4)]   {\rm (Sets of cardinality at least $\aleph_\alpha$ vs those less than $\aleph_\alpha$)} Suppose that $F$ satisfies $\#P_F(x_0,X;2)\ge \aleph_{\alpha}$ and  $ \#P_F(x,X;1)<\aleph_\alpha$  for
        all $x\ne x_0$ in $X$ and for some infinite cardinal number $\aleph_\alpha$.     Then $F$ has no iterative roots of order $n\ge 2$ in $\mathcal{F}^{-1}_{\aleph_{\alpha}}(X)$. If, in addition, $F\in \mathcal{F}^{-1}_{\aleph_{\alpha}}(X)$ and ${\rm Dom}(F)=X$, then $F$ has no iterative roots of order $n\ge 2$ at all.
    \end{enumerate}
\end{corollary}
As a consequence of the above corollary, analogous to the respective Corollaries \ref{C1} and \ref{C2} of Theorems \ref{No-root} and \ref{Gen-card}, we have the following.

\begin{corollary}\label{C4}
     Let $F$ be a multifunction on $X$ such that ${\rm Im}(F)=X$ and $x_0 \notin F(x_0)$ for some $x_0\in X$.
        \begin{enumerate}
        \item[\bf (1)]  {\rm (Large finite sets vs small finite sets)}
      Suppose that $F$ satisfies $\# F^2(x_0)>MN^3$  and   $\# F(x)\le N$  for
        all  $x\ne x_0$ in $X$ and for some  $N, M\in \mathbb{N}$.
            %
            %
        Then $F$ has no iterative roots
        of order $n\ge 2$ in $\mathcal{F}^{-1}_{M}(X)$.
        If, in addition, $F\in \mathcal{F}^{-1}_{M}(X)$ and ${\rm Dom}(F)=X$, then $F$ has no iterative roots of order $n\ge 2$ at all.

        \item[\bf (2)] {\rm (Infinite sets vs finite sets)}  Suppose that $F$ satisfies  $ F^{2}(x_0)$ is infinite and  $ F(x)$ is finite for
        all $x\ne x_0$ in $X$.
        Then $F$ has no iterative roots of order $n\ge 2$ in $\mathcal{F}^{-1}_{\textsf{f}}(X)$. If, in addition, $F\in \mathcal{F}^{-1}_{\textsf{f}}(X)$ and ${\rm Dom}(F)=X$, then $F$ has no iterative roots of order $n\ge 2$ at all.

        \item[\bf (3)] {\rm (Uncountable sets vs countable sets)}  Suppose that $F$ satisfies   $ F^{2}(x_0)$ is uncountable and  $ F(x)$ is countable for
        all $x\ne x_0$ in $X$.  Then $F$ has no iterative roots of order $n\ge 2$ in $\mathcal{F}^{-1}_{\textsf{c}}(X)$. If, in addition, $F\in \mathcal{F}^{-1}_{\textsf{c}}(X)$ and ${\rm Dom}(F)=X$, then $F$ has no iterative roots of order $n\ge 2$ at all.

            \item[\bf (4)] {\rm (Sets of cardinality at least $\aleph_\alpha$ vs those less than $\aleph_\alpha$)} Suppose that $F$ satisfies $ \#F^2(x_0)\ge \aleph_{\alpha}$ and   $ \#F(x)<\aleph_{\alpha}$  for
            all $x\ne x_0$ in $X$ and for some infinite cardinal number $\aleph_\alpha$.   Then $F$ has no iterative roots of order $n\ge 2$ in $\mathcal{F}^{-1}_{\aleph_{\alpha}}(X)$. If, in addition, $F\in \mathcal{F}^{-1}_{\aleph_{\alpha}}(X)$ and ${\rm Dom}(F)=X$, then $F$ has no iterative roots of order $n\ge 2$ at all.
    \end{enumerate}
\end{corollary}
The examples below illustrate the above results.
\begin{example}
    Let $F=F_1^{-1}$ (see Figure \ref{Fig4}), where $F_1$ is the map considered in Example \ref{E1}. Then it is easy to see that
            \begin{eqnarray*}
        \#P_F(X,x;1)=\left\{\begin{array}{cl}
        2&\text{if}~x\in \{x_{-2}^{(1)},x_{-2}^{(2)}\},\\
        1   &\text{otherwise}
        \end{array}\right. \quad \text{and}\quad    \#P_F(x,X;1)=\left\{\begin{array}{cl}
        4&\text{if}~x=x_0,\\
        1   &\text{otherwise}.
    \end{array}\right.
    \end{eqnarray*}
If we wish to apply result {\bf (1)} of Theorem \ref{No-root} to $F$, then we must consider $M$ and $N$ such that $M\ge 4$ and $N\ge 2$; but, in that case, there is no $x\in X$ such that the condition {\bf (1.a)} in result {\bf (1)} is satisfied. Therefore result {\bf (1)} of Theorem \ref{No-root} does not apply to $F$ directly.
Additionally, since $P_F(X,x;2)$ is finite for all $x\in X$, neither result {\bf (2)} nor {\bf (3)} of Theorem \ref{No-root} apply. Consequently, the results of Corollary \ref{C1} do not work either. However, $F$ has no iterative roots of order $n\ge 2$ by result {\bf (1)} of  Corollary \ref{C2}  with $N=1$ and $M=2$.
\end{example}
\begin{figure}[h]
    \centering
    \includegraphics[scale=0.6]{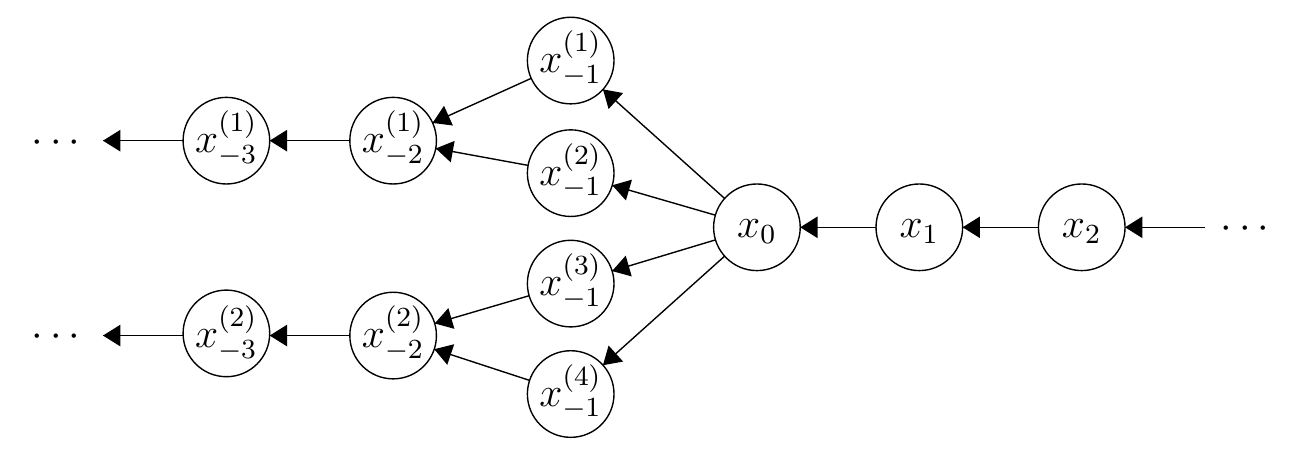}
    \caption{$F_1^{-1}$}
    \label{Fig4}
\end{figure}

\begin{example}
        {\rm Consider the multifunction $F:[0,1]\to 2^{[0,1]}$ defined by
            \begin{eqnarray*}
                F(x)=\left\{\begin{array}{cll}
                \{4x\}&\text{if}&x\in [0,\frac{1}{4}),\\
            \text{$[$}\frac{1}{2},\frac{3}{4}\text{$]$} &\text{if} &x=\frac{1}{4}, \\
                \{\frac{1}{3}(4x-1)\}   &\text{if}&x\in(\frac{1}{4},1]
                \end{array}\right.
            \end{eqnarray*}
            with exactly one set-value point $c=1/4$    (see Figure \ref{Fig5}). Then $F$ has no iterative roots of order $n\ge 2$ by result {\bf (2)} of  Corollary \ref{C3} (or Corollary \ref{C4}) with $x_0=1/4$.}
    \end{example}
    \begin{figure}[h]
\centering
\includegraphics[scale=1]{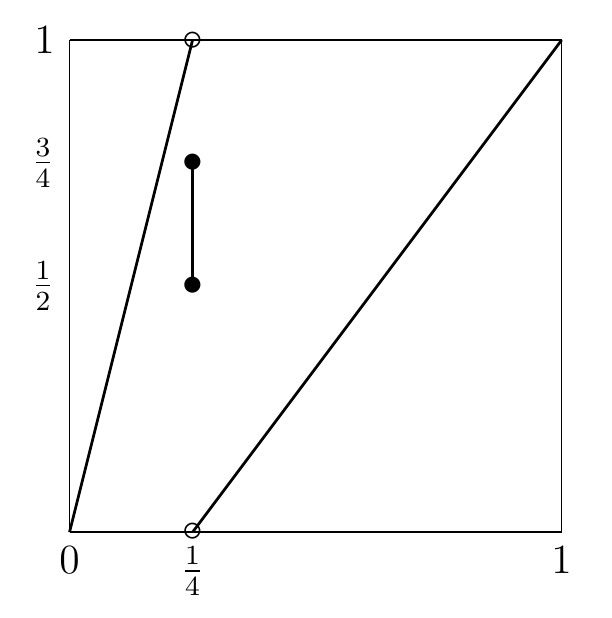}
    \caption{$F$}
    \label{Fig5}
\end{figure}

 So far, we have been looking at iterative roots of inverses of general multifunctions. We now restrict ourselves to a subclass of such multifunctions that arise from single-valued maps for further investigation. In what follows, let
 $\mathfrak{F}(X)$ denote the set of all maps $f:X\to X$, and for each set $A\subseteq X$, $f\in \mathfrak{F}(X)$ and $k\ge 1$, let  $f^{-k}(A)$ denote the $k$-th order {\it inverse image} of $A$ by  $f$ defined by   $f^{-k}(A)=\{x\in X:f^k(x)\in A\}$. An  $F\in \mathcal{F}(X)$ is said to be {\bf (i)} the {\it pullback} of a map $f\in \mathfrak{F}(X)$ if $F(x)=f^{-1}(\{x\})$ for all $x\in X$; {\bf (ii)} a {\it pullback multifunction} if it is the pullback of some $f\in \mathfrak{F}(X)$. Let $ \mathcal{F}_{\textsf{p}}(X)$ denote the set of all pullback multifunctions in $\mathcal{F}(X)$.
The following proposition provides a characterization of such
multifunctions.
\begin{proposition}\label{L1}
    Let $F\in \mathcal{F}(X)$ such that ${\rm Dom}(F)=X$. Then $F\in  \mathcal{F}_{\textsf{p}}(X)$  if and only if the following conditions are satisfied:
    \begin{enumerate}
        \item[\bf (a)] $F(x)\cap F(y)=\emptyset$ for $x\ne y$ in $X$;
        \item[\bf (b)] ${\rm Im}(F)=X$.
    \end{enumerate}
\end{proposition}
\begin{proof}
    Suppose that $F$ is the pullback of an $f\in \mathfrak{F}(X)$. Then $F$ clearly satisfies condition {\bf (a)} because $f^{-1}(\{x\})\cap f^{-1}(\{y\})=\emptyset$ for $x\ne y$ in $X$. Further, since $f:X\to X$ is a map, for each $x\in X$ there exists a unique $y\in X$ such that 
    $x\in f^{-1}(\{y\})=F(y)\subseteq F(X)$, implying that $X\subseteq F(X)$. The reverse inclusion follows trivially. Therefore $F$ satisfies condition {\bf (b)}.

    Conversely, assume that $F$ satisfies conditions {\bf (a)} and {\bf (b)}. Define a map $f:X\to X$ by $f(x)=y$ if $x\in F(y)$. Since for each $x\in X$ there exists a unique $y\in X$ such that $x\in F(y)$, where the existence is guaranteed  by condition {\bf (b)} and uniqueness by condition {\bf (a)},  $f$ is clearly  well-defined. Further, it is easy to check that $F(x)=f^{-1}(\{x\})$ for all $x\in X$. Therefore $F$ is the pullback of $f$.
\end{proof}

We remind that Theorem \ref{No-root} (resp. Corollary \ref{C1}) does not directly apply to pullback multifunctions $F$ because the assumption in this result that the number of $2$-paths ending at $x_0$ in $\mathcal{G}_F$ (resp. the number of points whose second order image contains $x_0$) is very large implies that the necessary condition {\bf (a)} of Proposition \ref{L1} for $F$ to be a pullback multifunction is not satisfied. However, we can certainly use Corollaries \ref{C3} and \ref{C4} for $F$,  the results derived from Theorem \ref{No-root} and Corollary \ref{C1}, respectively.
It is also worth noting that the class of pullback multifunctions is closed under the operation of taking iterative roots, as shown in the following.
\begin{proposition}\label{L2}
    Let $F\in   \mathcal{F}_{\textsf{p}}(X)$  such that $F(x)=(G_1\circ G_2)(x)$ for all $x\in X$ and some multifunctions $G_1, G_2\in \mathcal{F}(X)$, where ${\rm Dom}(F)={\rm Dom}(G_1)={\rm Dom}(G_2)=X$. Then {\bf (i)} ${\rm Im}(G_1)=X$ and {\bf (ii)} $G_2(x)\cap G_2(y)=\emptyset$ for $x\ne y$ in $X$. Further, if $F(x)=G^n(x)$ for all $x\in X$ and some multifunction $G\in \mathcal{F}(X)$ such that ${\rm Dom}(G)=X$, where $n\ge 2$, then $G\in     \mathcal{F}_{\textsf{p}}(X)$.
\end{proposition}

\begin{proof}
    Since $X=F(X)=G_1(G_2(X))\subseteq G_1(X)$, the result {\bf (i)} is trivial. If $z\in G_2(x)\cap G_2(y)$ for some $x\ne y$ in $X$, then $\emptyset \ne G_1(z)\subseteq F(x)\cap F(y)$, which  is a contradiction. Therefore result {\bf (ii)} follows. The second part follows from the first and the `if' part of Proposition \ref{L1}, because $F(x)=(G\circ G^{n-1})(x)=(G^{n-1}\circ G)(x)$ for all $x\in X$.
\end{proof}
The discussion just before Corollary \ref{C3} shows that a multifunction $F\in \mathcal{F}_{\textsf{p}}(X)$ with domain $X$, which is a pullback of an $f\in \mathfrak{F}(X)$, has an iterative root of order $n$ on $X$ if and only if $f$ has an iterative root of order $n$ in $\mathfrak{F}(X)$. More precisely, it is easy to check the following.
\begin{theorem}\label{pullback}
    Let $F\in \mathcal{F}_{\textsf{p}}(X)$ be the pullback of an $f\in \mathfrak{F}(X)$ such that ${\rm Dom}(F)=X$. Then $G$ is an iterative root of $F$ of order $n$ in $\mathcal{F}_{\textsf{p}}(X)$ if and only if $g$ is an iterative root of $f$ of order $n$ in $\mathfrak{F}(X)$, where $G$ is the pullback of $g$.
\end{theorem}
Thus, pullback multifunctions provide numerous examples of
multifunctions with or without iterative roots, based on the known
results on iterative roots for single-valued maps. An easy example
is that the $4$-th root multifunction $z\mapsto z^{1/4}$ has the
square root multifunction $z\mapsto z^{1/2}$ as its iterative square
root on $\mathbb{C}$ because the square function $z\mapsto z^2$ is
an iterative square root of the $4$-th power function $z\mapsto
z^4$.  
We now give a brief summary of known results on iterative roots of pullbacks of complex polynomials.
\begin{corollary}\label{C5}
    \begin{enumerate}
        \item[\bf (1)] The pullback of any complex quadratic polynomial has no iterative roots of order $n\ge 2$ on $\mathbb{C}$. In particular, the square root multifunction $z\mapsto z^{1/2}$ has no iterative roots of order $n\ge 2$ on $\mathbb{C}$.

        \item[\bf (2)] The $d$-th root multifunction $z\mapsto z^{1/d}$ has no iterative roots of order $n\ge 2$ on $\mathbb{C}$ provided $d^p \not\equiv d({\rm mod}~p^2)$ for all primes $p\le d$.  The first $25$ such $d$ are precisely $2, 3,6,11,14,15,34,39,47,58,59,66,83,86,87,95,102,103,106,111,114,119,123,139$ and $142$.

        \item[\bf (3)] Let $F$ be the pullback of a complex cubic  polynomial $f$, where $f$ is not linearly conjugate to $p(z)=z^3-z^2+z$ (i.e., $f=h\circ p\circ h^{-1}$ for no linear function $h(z)=\alpha z+\beta,~~\alpha\ne 0$ on $\mathbb{C}$) and has less than three distinct fixed points.  Then $F$ has no iterative roots of order $n\ge 2$ on $\mathbb{C}$.

        \item[\bf (4)] Let $F$ be the pullback of a complex polynomial  of degree $d\ge 2$.  Then $F$ has no iterative roots of order $n>d(d-1)$ on $\mathbb{C}$.

        \item[\bf (5)] Let $F$ be the pullback of a complex polynomial  of degree $d\ge 2$. If $n>d$ is a prime, then $F$ has no iterative roots of order $n$ on $\mathbb{C}$.
    \end{enumerate}
\end{corollary}
\begin{proof}
    Follows from the following known results on complex polynomials, respectively, and the
    discussions above.
    \begin{enumerate}
        \item[\bf (i)] (\cite[Theorem 1]{rice1980} or \cite[Theorem 2]{choczewski1992}) Complex quadratic polynomials have no iterative roots of any order $n\ge 2$ on $\mathbb{C}$. In particular, the square function $z\mapsto z^2$ has no iterative roots of any order $n\ge 2$ on $\mathbb{C}$.

        \item[\bf (ii)] (\cite{Solar1976}) The polynomial function $z\mapsto z^{d}$ has no iterative roots of order $n\ge 2$ on $\mathbb{C}$ provided $d^p \not\equiv d({\rm mod}~p^2)$ for all primes $p\le d$. The first $25$ such $d$ are those listed above in result {\bf (2)}, as shown in \cite{riesel1964}.

        \item[\bf (iii)] (\cite[Theorem 6]{choczewski1992}) Let $f$ be a complex cubic  polynomial which is not linearly conjugate to $p(z)=z^3-z^2+z$ and has less than three distinct fixed points.  Then $f$ has no iterative roots of order $n\ge 2$ on $\mathbb{C}$.

            \item[\bf (iv)] (\cite[Theorem 4]{rice1980})  Let $f$ be a complex polynomial of degree $d\ge 2$. Then $f$ has no iterative roots of order $n>d(d-1)$ on $\mathbb{C}$.

        \item[\bf (v)] (\cite[Theorem 1]{choczewski1992}) Let $f$ be a complex polynomial of degree $d\ge 2$. If $n>d$ is a prime, then $f$ has no iterative roots of order $n$ on $\mathbb{C}$.
    \end{enumerate}
\end{proof}
In this circle of ideas, we now make some general remarks, for which we require some additional terminology and notation.
 A fixed point $x\in X$ of an $f\in \mathfrak{F}(X)$ is said to be {\it isolated} if
 there is no $y\ne x$ in $X$ such that $f(y)=x$.
Let $T_{x}(f):=f^{-1}(\{x\})\setminus \{x\}$ for each fixed point
$x$ of $f$, with $T_{x}(f)$ being nonempty if and only if $x$ is
non-isolated. It is important to note that Theorem 4 in
\cite{rice1980} for complex polynomials, cited above is a
consequence of the general result in that paper, which is stated as
follows.
 \begin{theorem}{\rm (Lemma 11 in \cite{rice1980})} Let $f\in \mathfrak{F}(X)$ has a non-isolated fixed point $x\in X$ such that  $f^{-1}(\{y\})\ne \emptyset$ for some $y\in T_{x}(f)$. Then $f$ has no iterative roots of order $n>L$, where  $L:=\#(\cup_{x'}T_{x'}(f))$ with the union taken over all non-isolated fixed points $x'$ of $f$.
 \end{theorem}
\noindent  On the other hand, Theorem 1 in \cite{choczewski1992}, as seen in item {\bf (v)},
   is based on the degrees of the polynomials under consideration and applies only to iterative roots of prime orders.
We now prove a similar but more general result for general
single-valued maps based on the number of non-isolated fixed points.
It can also be applied to show the nonexistence of  iterative roots of
non-prime orders.

  \begin{theorem}\label{non-isolated}
        Let $f\in \mathfrak{F}(X)$ has $k\ge 2$ non-isolated fixed points $x_1, x_2,\ldots, x_k$ such that {\bf (a)} $\#T_{x_j}(f)\le l$ for all $1\le j\le k$, and {\bf (b)} $f^{-1}(\{y\})\ne \emptyset$ for all $y\in T_{x_j}(f)$ and $1\le j\le k$. Then $f$ has no iterative roots of order $n>l$ satisfying that $m\nmid n$ for all $2\le m\le k$.
  \end{theorem}
\begin{proof}
    Suppose, on the contrary, that $f=g^n$ for some $g\in \mathfrak{F}(X)$, where $n>l$ and $m\nmid n$ for all $2\le m\le k$. First, we assert that $g(x_j)$ is a non-isolated fixed point of $f$ for all $1\le j \le k$.

     Since $f(g(x_j))=g(f(x_j))=g(x_j)$, clearly  $g(x_j)$ is a  fixed point of $f$ for all $1\le j \le k$. Next, to show that they are non-isolated, suppose that $g(x_{j_0})$ is isolated for some $1\le j_0 \le k$, and let $y_{j_0}, z_{j_0}\in X$ be such that $y_{j_0}\in T_{x_{j_0}}(f)$ and $z_{j_0}\in f^{-1}(\{y_{j_0}\})$, both of which exist by {\bf (a)} and {\bf (b)}.  Then, as $g(x_{j_0})$ is isolated and
    \begin{align*}
        f^2(g(z_{j_0}))=g(f^2(z_{j_0}))=g(f(y_{j_0}))=g(x_{j_0}),
    \end{align*}
     we have $g(z_{j_0})=g(x_{j_0})$, implying that
     \begin{align*}
        y_{j_0}=f(z_{j_0})=g^{n-1}(g(z_{j_0}))=g^{n-1}(g(x_{j_0}))=f(x_{j_0})=x_{j_0}.
     \end{align*}
This contradicts our assumption that $y_{j_0}\in T_{x_{j_0}}(f)$. Therefore our assertion holds.

Now, since $x_1, x_2,\ldots, x_k$ are the only non-isolated fixed points of $f$ by hypothesis, we have
\begin{align}\label{g-invariant}
     g(\{x_1,x_2,\ldots, x_k\})\subseteq \{x_1,x_2,\ldots, x_k\}.
\end{align}
We discuss in the following two cases.

\noindent {\bf Case (i):} Suppose that $g(x_{j_0})=x_{j_0}$ for some $1\le j_0\le k$.
Without loss of generality, let $j_0=1$.
We first prove by induction that elements of $T_{x_1}(f)$ can be arranged in a sequence $y_1, y_2, \ldots, y_{l_1}$ such a way that
\begin{eqnarray}\label{g(Tx1)}
    g(y_i)\in \{x_{1}, y_1, y_2,\ldots, y_{i-1}\}~~ \text{for all}~~1\le i\le l_1,
\end{eqnarray}
 where $l_1:=\#T_{x_1}(f)$.

 Since $    f(g(y))=g(f(y))=g(x_1)=x_1$
 for all $y\in T_{x_1}(f)$, it is clear that
 \begin{align*}\label{gT}
    g(T_{x_1}(f))\subseteq T_{x_1}(f)\cup \{x_1\}.
 \end{align*}
Further, since $f=g^n$,  we cannot have $g(T_{x_1}(f))\subseteq T_{x_1}(f)$ (otherwise, $g^i(T_{x_1}(f))\subseteq T_{x_1}(f)$ for all $i\in \mathbb{N}$, and in particular $f(T_{x_1}(f))=g^n(T_{x_1}(f))\subseteq T_{x_1}(f)$, which is a contradiction to the definition of $T_{x_1}(f)$). Therefore there exists a $y_1\in T_{x_1}(f)$ such that $g(y_1)=x_1$, proving \eqref{g(Tx1)} for $i=1$. Next, assume that we have already defined $y_1, y_2, \ldots, y_i$ for some $1\le i\le l_1-1$. Then, using a similar argument as above, we see that
\begin{align*}
    g(T_{x_1}(f)\setminus\{y_1,y_2,\ldots y_i\})\nsubseteq T_{x_1}(f)\setminus\{y_1,y_2,\ldots y_i\}.
\end{align*}
Therefore there exists a $y_{i+1}\in T_{x_1}(f)\setminus\{y_1,y_2,\ldots y_i\}$ such that $g(y_{i+1})=x_1$, proving \eqref{g(Tx1)} for $i+1$. Thus, the existence of a desired rearrangement of $T_{x_1}(f)$ follows from induction.

Now, since $f^{-1}(\{y_1\})\ne \emptyset$ by {\bf (b)}, there exists a $z_1\in X$ such that $f(z_1)=y_1$.  Then
\begin{align*}
    f(g(z_1))=g(f(z_1))=g(y_1)=x_1,
\end{align*}
implying that $g(z_1)\in T_{x_1}(f)\cup \{x_1\}$. Also, from \eqref{g(Tx1)} we see that $g^i(y)=x_1$ for all $y\in T_{x_1}(f)\cup \{x_1\}$ and $i\ge l_1$.  Therefore,  as $n-1\ge l$ and $l\ge l_1$ by {\bf (a)}, we obtain
\begin{align*}
    y_1=f(z_1)=g^{n-1}(g(z_1))=x_1,
\end{align*}
which is a contradiction to our assumption on $y_1$ that $y_1\in T_{x_1}(f)$.

\noindent {\bf Case (ii):} Suppose that Case (i) does not hold. Then, as $g(x_{j})\ne x_{j}$ for all $1\le j\le k$, by using \eqref{g-invariant} we see that $g$ has a $q$-periodic point $x_{j_0}$ in $\{x_1,x_2,\ldots, x_k\}$ for some $1\le j_0\le k$ and $2\le q\le k$. As in Case (i), without loss of generality we assume that $j_0=1$. By our assumption on $x_1$, we have
\begin{eqnarray}\label{q-periodic}
    g^q(x_1)=x_1\quad \text{and}\quad g^i(x_1)\ne x_1~\text{for all}~1\le i\le q-1.
\end{eqnarray}
Further, since $2\le q\le k$, and $m\nmid n$ for all $2\le m\le k$ by hypothesis, we see from the  division algorithm that $n=sq+r$ for some $s\in \mathbb{N}$ and $1\le r\le q-1$.  Therefore, by \eqref{q-periodic} we have
\begin{eqnarray*}
    f(x_1)=g(x_1)=g^{sq+r}(x_1)=g^r(x_1)\ne x_1,
\end{eqnarray*}
which is a contradiction to our assumption that $x_1$ is a fixed point of $f$.

Thus, we get a contradiction in both the cases, proving that  $f$ has no iterative roots of order $n>l$ such that $m\nmid n$ for all $2\le m\le k$.
\end{proof}

It is worth noting that the assumption on $n$ that $n>l$ and $m\nmid n$ for all $2\le m\le k$ in the above theorem cannot be relaxed, as shown in the following.
\begin{example}
    {\rm  Let
        \begin{eqnarray*}
            X=\{x_{j}:1\le j\le 4\}\bigcup\{y_{j}^{(i)}:i=1,2~\text{and}~1\le j\le 4\}\bigcup\{z_{j}^{(i)}:i=1,2~\text{and}~1\le j\le 4\}
        \end{eqnarray*}
        and  $f:X\to X$ be defined by
        \begin{eqnarray*}
            f(x_j)&=&x_j\quad \text{for}~1\le j\le 4\\
            f(y_{j}^{(i)})&=&x_j\quad \text{for}~i=1,2~\text{and}~1\le j\le 4,\\
            f(z_{j}^{(i)})&=& y_{j}^{(i)} \quad \text{for}~i=1,2~\text{and}~1\le j\le 4
    \end{eqnarray*}
(see Figure \ref{Fig6}). Then it is clear that $x_1, x_2, x_3, x_4$ are the only  non-isolated fixed points of $f$, and $\#T_{x_j}(f)=2$ and $f^{-1}(\{y\})\ne \emptyset$ for all $y\in T_{x_j}(f)$ and $1\le j\le 4$. Therefore, it follows from  Theorem \ref{non-isolated} with $k=4$ and $l=2$ that $f$ has no iterative roots of order $n>l$ such that  $m\nmid n$ for all $2\le m\le k$. However, $f$ has an iterative root $g$ of order $n=4>l$ on $X$, given by
\begin{eqnarray*}
    g(x_j)&=& x_{j+1({\rm mod}~4)}\quad \text{for}~1\le j\le 4\\
g(y_j^{(i)})&=&y_{j+1}^{(i)}\quad \text{for}~i=1,2~ \text{and}~1\le j\le 3\\
g(y_{j}^{(i)})  &=& x_1 \quad \text{for}~i=1,2~\text{and}~j=4,\\
g(z_j^{(i)})&=&z_{j+1}^{(i)}\quad \text{for}~i=1,2~ \text{and}~1\le j\le 3,\\
g(z_{j}^{(i)})  &=& y_1^{(i)} \quad \text{for}~i=1,2~\text{and}~j=4
\end{eqnarray*}
(see Figure \ref{Fig7}), which is divisible by $m=2<k$. }
\end{example}
\begin{figure}[h]
    \centering
    \includegraphics[scale=0.6]{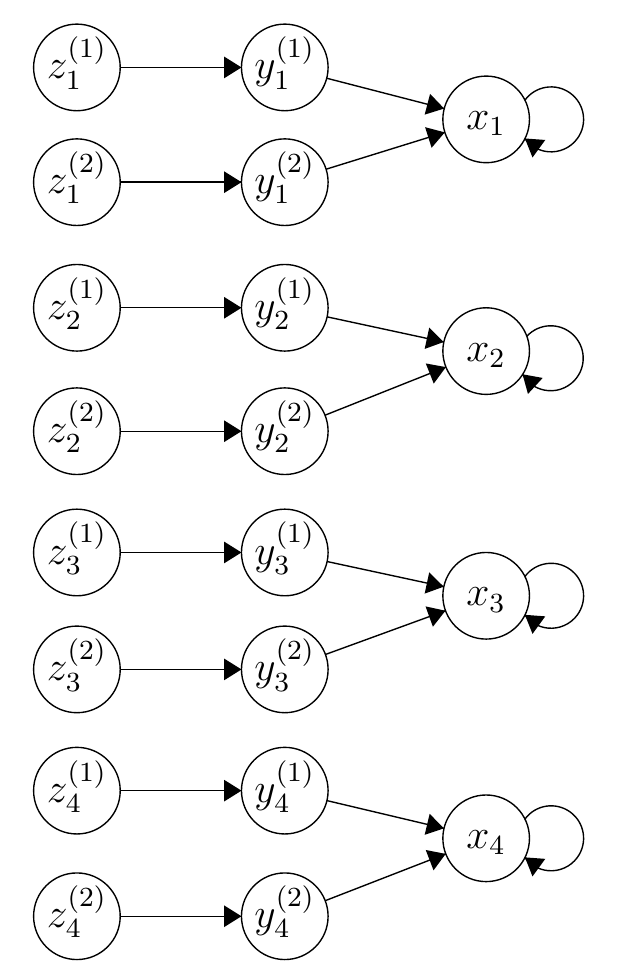}
    \caption{$f$}
    \label{Fig6}
\end{figure}
\begin{figure}[h]
    \centering
    \includegraphics[scale=0.6]{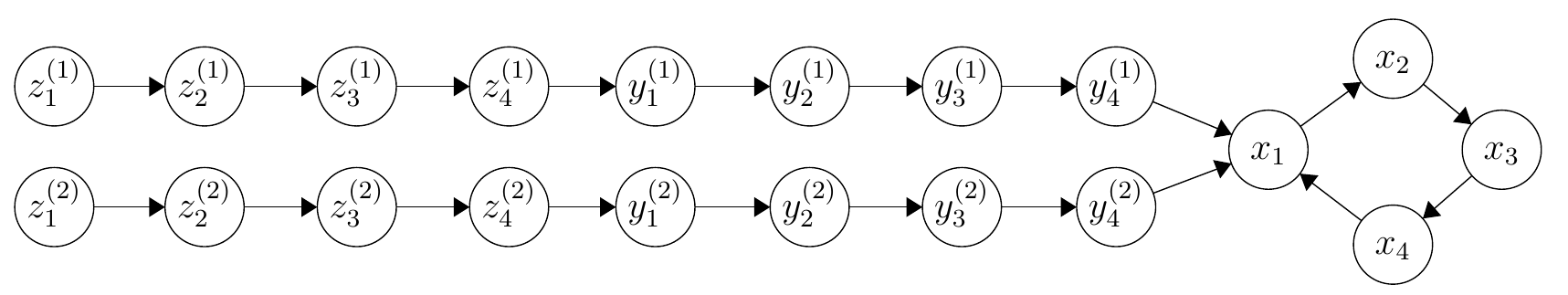}
    \caption{$g$}
    \label{Fig7}
\end{figure}

Since  every complex polynomial of degree $d\ge 2$ can have at most one isolated fixed point as seen from Lemma 4 in \cite{choczewski1992}, the only possible values for the pair $(k,l)$ considered in Theorem \ref{non-isolated} are $(d,d-1)$ and $(d-1,d-1)$. However, if $(k,l)=(d,d-1)$, then, in light of Theorem 4  in \cite{rice1980}, both Theorem \ref{non-isolated} and Theorem 1 in \cite{choczewski1992} provide the same results because every composite number $n$ between $k-1$ and $k(k-1)$ is divisible by $m$ for some $2\le m\le k$. Thus, the actual contribution of our Theorem \ref{non-isolated} to complex polynomials $f$ is limited to the case $(k,l)=(d-1,d-1)$, where $f$ has the form $\alpha(z-\beta)^d+\beta$ with $\alpha\ne 0$.
In fact, Theorem \ref{non-isolated} implies that {\it every $p$-th power has no $p$-th roots for all primes $p$}, which is not guaranteed by either Theorem 4 in \cite{rice1980} or Theorem 1 in \cite{choczewski1992}. More precisely, we have the following.

\begin{corollary}\label{prime}
If $d\ge 2$ is a prime, then the complex polynomial function $z\mapsto \alpha (z-\beta)^d+\beta$ with $\alpha\ne 0$ has no iterative roots of order $d$ on $\mathbb{C}$.
\end{corollary}
\begin{proof}
    Let $f(z)=\alpha(z-\beta)^d+\beta$ for all $z\in \mathbb{C}$, where $\alpha\ne 0$ and $d\ge 2$ is a prime. Then $f$ has $d-1$ non-isolated fixed points, which are precisely  $\beta+\alpha^{-1/d-1}$, where $\alpha^{-1/d-1}$ denote any of the $d-1$ possible values. Therefore we see from Theorem \ref{non-isolated} that $f$ has no iterative roots of order $n>d-1$ on $\mathbb{C}$ with $m\nmid n$ for all $2\le m\le d-1$.  In particular, as $d\ge 2$ is a prime, it follows that $f$ has no iterative roots of order $d$ on $\mathbb{C}$.
\end{proof}

The following example illustrates the above corollary.

\begin{example}
    {\rm Consider the polynomial function $f(z)=z^5$ on $\mathbb{C}$, to which the result in \cite{Solar1976} does not apply, as seen in item {\bf (ii)} of the proof of Corollary \ref{C5}. Then it follows from the above corollary that $f$ has no iterative roots of order $5$ on $\mathbb{C}$, which is not guaranteed by Theorem 1 in \cite{choczewski1992}. }
\end{example}

Finally, we remind that, while  Theorem \ref{non-isolated} and Corollary \ref{prime} are concerned with single-valued maps, we can use Theorem \ref{pullback} to obtain similar results for their pullback multifunctions. It is worth noting that these results differ from those obtained for general multifunctions in Section \ref{Sec2} in the following sense: The former is based on a fixed point of the map under consideration, whereas the latter is based on a point that is not a fixed point.

\bibliographystyle{amsplain}

\begin{thebibliography}{10}


\bibitem{Abel}
N. H. Abel, Oeuvres Complètes, T.II, {\it Christiania}, (1881), 36--39.


\bibitem{Aubin-Frankowska}
J. P. Aubin, H. Frankowska, {\it Set-Valued Analysis}, Birkh\"{a}user, Boston, Basel,
Berlin, 1990.

\bibitem{Babbage1815}
C. Babbage, Essay towards the calculus of functions, {\it Philos. Trans.}, (1815), 389--423.
%
\bibitem{Baron-Jarczyk}
K. Baron, W. Jarczyk, Recent results on functional equations in a single variable,
perspectives and open problems, {\it Aequationes Math.},  61 (2001), 1--48.


\bibitem{BG}
B. V. R. Bhat, C. Gopalakrishna, Iterative square roots of functions, {\it Ergodic Theory Dynam. Systems}, (2022), pp. 1--27. 

\bibitem{BGpreprint}
B. V. R. Bhat, C. Gopalakrishna, The non-iterates are dense in the space of continuous self-maps, \url{https://arxiv.org/abs/2208.04093}.


\bibitem{Blokh}
A. M. Blokh, The set of all iterates is nowhere dense in
{$C([0,1],[0,1])$}, {\it Trans. Amer. Math. Soc.},  333 (1992), 2,  787--798.



\bibitem{choczewski1992}
B. Choczewski and M. Kuczma, On iterative roots of polynomials, {\it European Conference
on Iteration Theory (Lisbon, 1991)}, 59–67, World Sci. Publishing, Singapore, 1992.


%
\bibitem{Edgar}
G. A. Edgar, Fractional iteration of series and transseries, {\it Trans. Amer. Math. Soc.},  365 (2013), 11,  5805--5832.

\bibitem{fort1955}
M. K. Fort Jr, The embedding of homeomorphisms in flows, {\it Proc. Amer. Math. Soc.},  6 (1955),
960--967.

\bibitem{Hu-Papageorgiou}
S. Hu, N.S. Papageorgiou, {\it Handbook of Multivalued Analysis}, Kluwer Academic, Dordrecht, 1997.

\bibitem{Humke-Laczkovich1989}
P. D. Humke, M. Laczkovich, The Borel structure of iterates of continuous functions, {\it Proc. Edinburgh Math. Soc.},  32 (1989), 483--494.


\bibitem{Iannella}
N. Iannella and L. Kindermann, Finding iterative roots with a spiking
neural network, {\it Inform. Process. Lett.}, 95 (2005), 545--551.
%
%

\bibitem{Jarczyk-Powierza}
W. Jarczyk, T. Powier\.{z}a, On the smallest set-valued iterative roots of bijections, Dynamical Systems and Functional Equations (Murcia, 2000), Internat. {\it J. Bifur. Chaos Appl. Sci. Engrg.} 13 (2003), 1889--1893.

\bibitem{Jarczyk-Zhang}
W. Jarczyk, W. Zhang, Also set-valued functions do not like
iterative roots, {\it Elemente Math.}, 62 (2007),  1--8.
%
\bibitem{Kindermann}
L. Kindermann, Computing iterative roots with neural networks,
{\it Proc. Fifth Conf. Neural Info. Processing} 2 (1998), 713--715.
%
\bibitem{Konigs}
G. K\"{o}nigs, Recherches sur les int\'{e}grales de certaines \'{e}quations fonctionnelles,
{\it Ann. Ecole Norm. Sup.} (3) (1884), Suppl.  1, 3--41.
%
%
\bibitem{Kuczma1968}
M. Kuczma, {\it Functional Equations in a Single Variable}, Polish Scientific, Warsaw, 1968.

%
\bibitem{kuczma1990}
M. Kuczma, B. Choczewski, R. Ger,  {\it Iterative functional equations}, volume
32 of Encyclopedia of Mathematics and its Applications. Cambridge University
Press, Cambridge, 1990.

%
%
%
%
\bibitem{li2009}
L. Li, J. Jarczyk, W. Jarczyk, W. Zhang, Iterative roots of mappings with a unique set-value point,
{\it Publ. Math. Debrecen},  75 (2009), 203--220.


%
%
\bibitem{li2012}
L. Li, W. Zhang, Construction of usc solutions for a multivalued iterative equation of order
$n$, {\it Results Math.},  62 (2012), 203--216.
%
\bibitem{LiZhang2018}
L. Li, W. Zhang, The question on characteristic endpoints for iterative roots of PM functions, {\it J. Math. Anal. Appl.},   458  (2018), 1,  265--280.

\bibitem{Lin2014}
Y. Lin, Existence of iterative roots for the sickle-like functions, {\it J.
    Inequal. Appl.}, (2014), Paper No. 204, 23 pp.
%
\bibitem{Lin-Zeng-Zhang2017}
Y. Lin, Y. Zeng, W. Zhang, Iterative roots of clenched single-plateau
functions, {\it Results Math.}, 71 (2017),  1-2, 15--43.



%
%
\bibitem{LiuLiZhang2018}
L. Liu, L. Li, W. Zhang, Open question on lower order iterative roots for PM functions, {\it J. Difference Equ. Appl.},   24    (2018), 5, 825--847.
%
\bibitem{LiuLiZhang2021}
L. Liu, L. Li, W. Zhang, Iterative roots of exclusive multifunctions,
{\it J. Difference Equ. Appl.} 27 (1) (2021), 41--60.

%
\bibitem{Liu-zhang2021}
L. Liu, W. Zhang, Genetics of iterative roots for PM functions, {\it Discrete Contin. Dyn. Syst.}, 41 (2021),  5, 2391--2409.

\bibitem{Lydzinska2018}
G. \L{}ydzi\'{n}ska, On iterative roots of order $n$ of some multifunctions with a unique set-value point, {\it Publ. Math. Debrecen}, 93 (2018), 1--8.

%
%
%

\bibitem{Nan-li-2011}
Z. Nan and L. Li, On square iterative roots of multifunctions, {\it Acta Math. Univ. Comenian.
    (N.S.)}, 80 (2011), 39--42.

%

\bibitem{Powierza1999}
T. Powier\.{z}a, Set-valued iterative square roots of bijections, {\it Bull. Pol. Ac.: Math.} 47 (1999), 377--383.

\bibitem{Powierza2001}
T. Powier\.{z}a, Strong set-valued iterative roots, {\it Grazer Math. Ber.} 344 (2001), 51--56.

\bibitem{Powierza2002}
T. Powier\.{z}a, Higher order set-valued iterative roots of bijections, {\it Publ. Math. Debrecen}, 61 (2002), 315--324.


\bibitem{riesel1964}
H. Riesel, Note on the congruence $a^{p-1}=1({\rm mod}~p^2)$, {\it Math. Comp.}, 18 (1964), 149--150.

\bibitem{rice1980}
R. E. Rice, B. Schweizer, A. Sklar, When is $f(f(z)) = az^2 + bz + c?$, {\it Amer. Math. Monthly},  87 (1980), 4,  252--263.
%

\bibitem{simon1989}
K. Simon, Some dual statements concerning Wiener measure and Baire category, {\it Proc. Amer. Math. Soc.},  106 (1989), 2, 455--463.




\bibitem{Smajdor1985}
A. Smajdor, {\it Iterations of multivalued functions}, Pr. Nauk. Uniw. \'{S}lask. Katowicach, 759 (1985), 1--59.

\bibitem{Solar1976}
A. Solar,  On the talk of Miss Riggert, {\it Aequationes    Math.}, 14 (1976), 226.

%
%
\bibitem{Targonski1981}
G. Targonski, {\it Topics in Iteration Theory}, Vandenhoeck and Ruprecht, G\"{o}ttingen, 1981.

\bibitem{Targonski1995}
G. Targonski, Progress of iteration theory since 1981, {\it Aequationes Math.},  50 (1995),  50--72.

%

\bibitem{Xu-Nikodem-Zhang}
B. Xu, K. Nikodem, W. Zhang, On a multivalued iterative equation of order $n$, {\it J. Convex Anal.}, 18 (2011), 673--686.

%
\bibitem{Yu2021}
Z. Yu, L. Li, L. Liu, Topological classifications for a class of 2-dimensional quadratic mappings and an application to iterative roots,
{\it Qual. Theory Dyn. Syst.}, 20 (2021), 1, Paper No. 2, 25 pp.
%
%
%
\bibitem{zdun-soalrz}
M. C. Zdun, P. Solarz, Recent results on iteration theory: iteration groups and semigroups in
the real case, {\it Aequationes Math.},  87 (2014), 201--245.


\bibitem{Zhang-Huang}
P. Zhang, L. Huang, Iterative roots of upper semicontinuous
multifunctions, {\it Adv. Difference Equ.}, (2017), Paper No. 367, pp. 14.
%
%
%
%
%



%
%
%
%


\end{thebibliography}

\end{document}